\pdfoutput=1
\documentclass[times,letter]{article}
\usepackage{graphicx}
\usepackage[comma,square,numbers]{natbib}

\usepackage{array}
\usepackage{amsmath}
\usepackage{amssymb}  
\usepackage{amsthm}
\usepackage{savesym}
\usepackage{amsfonts}
\savesymbol{iint}
\usepackage{txfonts}
\restoresymbol{TXF}{iint}

\newcommand{\set}[1]{\mathcal{#1}}
\newcommand{\R}{\mathbb{R}}

\DeclareMathOperator{\diag}{diag}

\theoremstyle{plain}
\newtheorem{thm}{Theorem}[section]
\newtheorem{lem}[thm]{Lemma}

\newtheorem{cor}[thm]{Corollary}

\theoremstyle{definition}

\newtheorem{exmp}{Example}[section]

\theoremstyle{remark}

\newtheorem{rmk}{Remark}[section]

\usepackage{hyperref}
\hypersetup{
  pdfinfo={
    Title={Generalized Sharp Bounds on the Spectral Radius of Digraphs},
    Author={Brian K. Butler},
    Subject={Spectrum of Nonnegative Matrices},
    Keywords={Digraph; Nonnegative matrix; Spectral radius}
  }
}

\begin{document}

%
%
%
%
%

\title{Generalized Sharp Bounds on the Spectral Radius of Digraphs}
\author{Brian~K.~Butler and Paul~H.~Siegel \\Department of Electrical and Computer Engineering, \\University of California San~Diego, La~Jolla, CA 92093 USA}
\date{Nov 30, 2012}
\maketitle

\begin{abstract}
The spectral radius $\rho(G)$ of a digraph $G$ is the maximum modulus of the eigenvalues of its adjacency matrix.  We present bounds on $\rho(G)$ that are often tighter and are applicable to a larger class of digraphs than previously reported bounds. Calculating the final bound pair is particularly suited to sparse digraphs.

For strongly connected digraphs, we derive equality conditions for the bounds, relating to the outdegree regularity of the digraph. We also prove that the bounds hold with equality only if $\rho(G)$ is the $r$-th root of an integer, where $r$ divides the index of imprimitivity of $G$.
\end{abstract}

\section{Introduction}
\label{sectintro}
Let $A=(a_{ij})$, $1\le i,j \le n$ be an $n \times n$ matrix over the complex numbers. 
The eigenvalues of $A$ are the complex roots of the characteristic equation $\det(A-\mu I)=0$.
The set of distinct eigenvalues is called the \emph{spectrum} of $A$, denoted $\sigma(A) = \{\mu_1,\ldots,\mu_m\}$,
and the \emph{spectral radius} of $A$ is the real number $\rho(A) = \max\left\{|\mu| : \mu \in \sigma(A)\right\}$.
The matrix $A$ is said to be \emph{nonnegative}, denoted $A\ge0$, if every entry $a_{ij}$ is real and greater than or equal to zero.
Specializing to vectors, we say that the complex column vector $x=(x_1,\ldots,x_n)^T$ is \emph{nonnegative}, or $x \ge 0$, if every element is real and greater than or equal to zero. If every element $x_i$ is real and strictly greater than zero, we say that $x$ is \emph{positive}, or $x > 0$.
It is well known (see \citet[p. 503]{Horn}) that the spectral radius of a nonnegative matrix $A$ is an eigenvalue; that is, $\rho(A) \in \sigma(A)$.

Now, let $G=(V,E)$ be a directed graph, or \emph{digraph}, defined by the vertex set $V=\{v_1, \ldots, v_n\}$ and the collection $E$ of directed edges,
or \emph{arcs}, between ordered pairs of vertices. 
The \emph{adjacency matrix} of $G$, denoted $A(G)$, is the nonnegative matrix whose $(i,j)$-th entry is the number of arcs from vertex $v_i$ to vertex $v_j$.  
The \emph{spectral radius} $\rho(G)$ of the digraph $G$ is defined to be the spectral radius of $A(G)$.

In this paper, we derive several new bounds on the spectral radius of nonnegative matrices which we then use to bound the spectral radius of a large class of digraphs.
Our results generalize those found in \citet{Zhang}, \citet{Kolo}, \citet{Xu}, and \citet{Gungor}.  
With respect to the bounds of \citet{Liu}, we find new equality conditions when they are applied to digraphs.
(For a recent survey on prior work in this area, see \citet{Brualdi2010}.)

The nonnegative $n\times n$ matrix $A$, with $n\ge 2$, is said to be \emph{reducible} if there exists a permutation matrix $P$ 
such that $P A P^T = \bigl(\begin{smallmatrix} X&Y\\0&Z\end{smallmatrix}\bigr)$
where $X$ and $Z$ are square submatrices. Otherwise, $A$ is said to be \emph{irreducible}.
Let $r_i(A)$ denote the sum of the elements along the $i$-th row of $A$; that is
\begin{equation*}
r_i(A)=\sum_{j=1}^n a_{ij}
\end{equation*}
for $i\in \{1,\ldots,n\}$.
The following classical result gives a bound on the spectral radius of the nonnegative matrix $A$ in terms of its row sums.


\begin{thm}[Frobenius] 
\label{thmF}
Let $A=(a_{ij})$ be an $n \times n$ nonnegative matrix with spectral radius $\rho(A)$ and row sums $r_i(A)$, $i\in \{1,\ldots,n\}$.
Then
\begin{equation}
\label{eqF}
\min_i r_i(A)  \le \rho(A) \le \max_i  r_i(A).
\end{equation}
Moreover, if $A$ is an irreducible matrix, then equality holds on either side (and hence both sides) of \eqref{eqF} if and only if all row sums of $A$ are equal. 
\end{thm}
\begin{proof}
See \citet[pp. 24--26]{Minc}. 
\end{proof}
\begin{rmk}
Theorem~\ref{thmF} also applies to column sums since $\rho(A^T)=\rho(A)$.
\end{rmk}

Again the following bound on the spectral radius is well known \cite{Horn,Zhang}.
In the proof we use the concept of the \emph {sparsity pattern} of a complex matrix, which refers to the locations of its nonzero entries.

\begin{thm}
\label{thmW}
Let $A$ be an $n \times n$ nonnegative matrix with spectral radius $\rho(A)$ and $x=(x_1,\ldots,x_n)^T$ be a positive column vector. Then
\begin{equation}
\label{eqW}
\min_{1\le i\le n} \left[ \frac{(Ax)_i}{x_i} \right] \le \rho(A) \le \max_{1\le i\le n} \left[ \frac{(Ax)_i}{x_i} \right].
\end{equation}
Moreover, if $A$ is an irreducible matrix, then equality holds on either side (and hence both sides) of \eqref{eqW} if and only if
the vector $x$ is an eigenvector of $A$ corresponding to $\rho(A)$.
\end{thm}
\begin{proof}
By the assumption that $x_i>0,$ for all $i\in \{1, \ldots, n\}$, the diagonal matrix $D=\diag(x_1,\ldots,x_n)$ is invertible.
Since $A$ and $D^{-1}A D$ are similar matrices, they have identical eigenvalues, and therefore $\rho(A) = \rho(D^{-1}A D)$.
The row sums of $D^{-1}A D$ are given by
\begin{equation}
\label{eqW1}
r_i \left(D^{-1}A D\right)=\sum_{j=1}^n \frac{a_{ij} x_j}{x_i} = \frac{(A x)_i}{x_i},
\end{equation}
for all $i \in \{1,\ldots,n\}$.
Thus, Theorem~\ref{thmF}, with $D^{-1}A D$ substituted for $A$, implies \eqref{eqW}.

Since $A$ and $D^{-1}A D$ have identical sparsity patterns, $A$ is irreducible if and only if $D^{-1}A D$ is irreducible.
Therefore, if either equality in \eqref{eqW} holds, then by the equality condition of Theorem~\ref{thmF}, \eqref{eqW1} equals $\rho(A)$ for all $i$,
which yields $A x = \rho(A) x$, as desired.
Conversely, if $x>0$ and  $A x = \rho(A) x$, then the row sums in \eqref{eqW1} are equal to $\rho(A)$ for all $i$, forcing equality on both sides of \eqref{eqW}.
\end{proof}

\section{Spectral Bounds for Nonnegative Matrices} 
\label{sectMatrix}
In this section, we characterize the spectral radius of nonnegative matrices with nonzero row sums.
It is well known \cite{Xu,Liu} that deleting the zero rows and their corresponding columns 
(\textit{i.e.}, the columns having the same indices as the zero rows) 
leaves unaffected the nonzero entries in the spectrum of a matrix.  
Since the column removal may reveal new all-zero rows, this process may have to be applied multiple times to finally 
produce a matrix with nonzero row sums.  
Once this is achieved, the bounds of this section may be applied to the reduced matrix.

Let $A=(a_{ij})$ be an $n \times n$ matrix.  
We denote the $(i,j)$th entry of matrix $A^k$ by $a_{ij}^{(k)}$, noting that
\begin{equation*}
a_{ij}^{(k)}= \overbrace{\sum_{s=1}^n\sum_{t=1}^n\cdots \sum_{y=1}^n}^{k-1 \text{ sums}} 
\overbrace {a_{is}a_{st}\cdots a_{yj}}^{k \text{ terms}}\quad\text{ and }\quad a_{ij}^{(0)}=\delta_{ij},
\end{equation*}
where $\delta_{ij}$ is the Kronecker delta.   
Let $r_i(A^k)$ denote the sum of the $i$th row of $A^k$, that is, $r_i(A^k)= \sum_{j=1}^n a_{ij}^{(k)}$.  
Using the fact that, for any $n \times n$ matrix $B$, the row sums of the product $AB$ are given by
\begin{equation}
\label{genRS}
r_i(AB)=  \sum_{j=1}^n \sum_{k=1}^n a_{ij} b_{jk} = \sum_{j=1}^n a_{ij} r_j(B),
\end{equation}
we can derive additional useful row-sum expressions such as 
\begin{equation}
\label{rowsum_order_k}
r_i(A^k)= \sum_{j=1}^n a_{ij}^{(k-t)}r_j(A^t),
\end{equation}
for all $0\le t\le k$.  
We will make frequent use of the column vector $x=(r_1(A^k),\ldots,r_n(A^k))^T$ and the diagonal matrix $D=\diag(r_1(A^k),\ldots,r_n(A^k))$, 
for some integer $k\ge0$.
Then, derived from \eqref{rowsum_order_k},  
\begin{equation}
\label{usefulx}
 (A^t x)_i = \sum_{j=1}^n a_{ij}^{(t)}r_j(A^k) = r_i(A^{t+k})
\end{equation}
and, assuming that the row sums of $A^k$ are nonzero,
\begin{equation}
\label{usefulD}
 r_i(D^{-1} A^t D) = \frac{\sum_{j=1}^n a_{ij}^{(t)}r_j(A^k)}{r_i(A^{k})} = \frac{r_i(A^{t+k})}{r_i(A^{k})},
\end{equation}
for any $t \ge 0$ and all $i \in \{1,\ldots,n\}$. 
Also, as \citet{Liu} remarked, if the row sums of a nonnegative matrix $A$ are nonzero, then so are the row sums of $A^k$, for $k\ge0$.
After the following lemma, we will use these high-order row sums to bound the spectrum of $A^L$, for any $L>0$. 
Recall from Perron-Frobenius theory that irreducible matrices have a unique (up to a scale factor) positive eigenvector corresponding to the spectral radius. 

\begin{lem} 
\label{irred-red}
Let $A$ be an $n \times n$ nonnegative matrix with spectral radius $\rho(A)$.
If $A^L$ is irreducible, for some $L> 0$, then $A$ is also irreducible and
the positive eigenvectors of $A$ and $A^L$ agree up to a scale factor.
\end{lem}
\begin{proof}
First take $A$ to be reducible.
In this case, there exists a permutation matrix such that $P A P^T = \bigl(\begin{smallmatrix} X&Y\\0&Z\end{smallmatrix}\bigr)$.
Clearly $P A^L P^T$ is also upper triangular, and hence $A^L$ is reducible too (a contradiction).

Since $A$ is irreducible, it has a unique positive eigenvector, which we denote $x$, corresponding to $\rho(A)$.
Repeatedly left-multiplying $A x=\rho(A) x$ by $A$ implies that $A^{t} x = \rho(A)^{t} x = \rho(A^t) x$, for any integer $t\ge 0$.
In particular, setting $t=L$, we conclude that $x$ is the unique (up to a scale factor) positive eigenvector of $A^L$
\end{proof}

\begin{thm}[Liu \cite{Kolo,Liu}]
\label{thmM1}
Let $A$ be an $n \times n$ nonnegative matrix with spectral radius $\rho(A)$ and row sums $r_1(A),\ldots,r_n(A)$, all nonzero. 
Then, for any integers $L>0$ and $k \ge 0$,
\begin{equation}
\label{eqM1}
\min_{1\le i\le n}  \left( {\frac{r_i\left(A^{k+L}\right)}{r_i(A^k)}} \right)^{1/L}\mspace{-10.0mu} \le \rho(A) \le 
\max_{1\le i\le n} \left( {\frac{r_i\left(A^{k+L}\right)}{r_i(A^k)}} \right)^{1/L}\mspace{-10.0mu}.
\end{equation}
Moreover, if $A^L$ is an irreducible matrix,
then equality holds on either side (and hence both sides) of \eqref{eqM1} if and only if 
$x=(r_1(A^k),\ldots,r_n(A^k))^T$ is an eigenvector of $A$. 
\end{thm}
\begin{proof}  
By successively left-multiplying $A x = \rho(A) x$ by $A$ is easy to show that $\rho(A^L)=\rho(A)^L$.
Let $x=(r_1(A^k),\ldots,r_n(A^k))^T$. 
Then, by \eqref{usefulx}, we know $(A^L x)_i = r_i(A^{k+L})$,
and we may confirm \eqref{eqM1} by applying Theorem~\ref{thmW} to $A^L$ and $x$.

Now assume that $A^L$ is irreducible.
The application of Theorem~\ref{thmW} provided the equality condition that $x$ be an eigenvector of $A^L$.
Finally, applying Lemma~\ref{irred-red}, shows that equality holds if and only if $x$ is an eigenvector of $A$.
\end{proof}
\begin{rmk}
The nonzero row assumption is not needed when $k=0$, since $r_i(A^0)=1$ and $x$ will still be a positive vector in this case.
As discussed later, these bounds are more general than those of \citet{Zhang}.
\end{rmk}

The following theorem presents another useful result from Liu \cite{Liu} that we will use later.
It shows that, as functions of the index $k$, the upper and lower bounds of \eqref{eqM1} form monotonically non-increasing and non-decreasing sequences, respectively.

\begin{thm} 
\label{thmMono}
Let $A$ be an $n \times n$ nonnegative matrix with nonzero row sums.
Then, for any integer $L>0$,
\begin{equation*}
\min_{1\le j\le n} \left[ \frac{r_j\left(A^{k+L}\right)}{r_j(A^k)}  \right]\le
                 \frac{r_i\left(A^{k+1+L}\right)} {r_i(A^{k+1})}\le
\max_{1\le j\le n} \left[ \frac{r_j\left(A^{k+L}\right)}{r_j(A^k)}  \right] ,
\end{equation*}
for all $k\ge0$ and $i \in \left\{ 1,\ldots,n \right\}$ . 
\end{thm}
\begin{proof}
See \citet[Theorem~3.3]{Liu}.
\end{proof}

The following theorem provides a generalization of the bounds in \citet{Xu}, but first we present a useful lemma.

\begin{lem} 
\label{lemEqEig}
Let $A$ be an $n \times n$ matrix with spectral radius $\rho(A)\ne 0$.   
If, for some $k\ge0$, $x=(r_1(A^k),\ldots,r_n(A^k))^T$ is an eigenvector of $A$ corresponding to $\rho(A)$, then so is 
$y=(r_1(A^{k+1}),\ldots,r_n(A^{k+1}))^T$, and $y =\rho(A) x$.
\end{lem}
\begin{proof}
Since $(A x)_i = \sum_{j=1}^n a_{ij} r_j(A^k) =  r_i(A^{k+1}) = y_i$ for all $i\in \{1,\ldots,n\}$, then $A x = y = \rho(A) x$.
\end{proof}

\begin{thm}
\label{thmXuXuMat}
Let $A$ be an $n \times n$ nonnegative matrix with spectral radius $\rho(A)$ and nonzero row sums. 
Then, for any integers $M>0$, $N\ge0$, and $k\ge0$,
\begin{equation}
\label{eqX1}
\min_{\substack{1\le i\le n\\1\le j\le n}} \left[ \left( {\frac{r_i\left(A^{k+M}\right) r_j\left(A^{k+N}\right)}{r_i(A^k) \, r_j(A^k)}} \right)^{\frac{1}{M+N}}\mspace{-12.0mu} : a_{ij}^{(M)}>0 \right]
\le \rho(A) \le 
\max_{\substack{1\le i\le n\\1\le j\le n}} \left[ \left( {\frac{r_i\left(A^{k+M}\right) r_j\left(A^{k+N}\right)}{r_i(A^k) \, r_j(A^k)}} \right)^{\frac{1}{M+N}}\mspace{-12.0mu} : a_{ij}^{(M)}>0 \right].
\end{equation}
Moreover, if $A^{M+N}$ is an irreducible matrix, then equality holds in either side (and hence both sides) of \eqref{eqX1} if and only if
$x=(r_1(A^k),\ldots,r_n(A^k))^T$ is an eigenvector of $A$. 
\end{thm}
\begin{proof}
Define the invertible diagonal matrix $D=\diag(r_1(A^k),\ldots,r_n(A^k))$.
Since $A^{M+N}$ and $D^{-1}A^{M+N} D$ are similar matrices, we have 
\begin{align}
\label{2.4.0}
\rho(A^{M+N}) &=  \rho(D^{-1}A^{M+N} D)\\ 
\label{2.4.1}
                & \le \max_{1\le i \le n}  r_i \left(D^{-1}A^{M+N} D\right) 
\end{align}
where the inequality in (\ref{2.4.1}) follows from Theorem~\ref{thmF}.
The row sums on the right hand side of (\ref{2.4.1}) can be formulated as
\begin{align}
\notag
r_i(D^{-1}A^M A^N D) &= r_i \left(\left(D^{-1}A^M D\right) \left(D^{-1} A^N D\right) \right) \\ 
\label{2.4.2}
         &=\sum_{j=1}^n \left(D^{-1}A^M D\right)_{ij} r_j \left(D^{-1}A^N D\right)\\
\label{2.4.3}
 	&\le   r_i \left(D^{-1}A^M D\right)  \max_{j} \left\{ r_j \left(D^{-1}A^N D\right) : a^{(M)}_{ij} > 0 \right\}\\ 
\label{2.4.4}
	&= \max_{j} \left\{ r_i \left(D^{-1}A^M D\right) r_j \left(D^{-1}A^N D\right) : a^{(M)}_{ij} > 0 \right\},
\end{align}
where (\ref{2.4.2}) follows from (\ref{genRS}).  The restricted maximizations in  (\ref{2.4.3}) and (\ref{2.4.4}) make use of the fact that the sparsity patterns of $A^M$ and $D^{-1}A^M D$ are the same.
Applying \eqref{usefulD} to the factors in the product in (\ref{2.4.4}), with $t=M$ and $t=N$, respectively, we conclude that
\begin{equation*}
\rho(A^{M+N}) \le
\max_{\substack{1\le i\le n\\1\le j\le n}} \left[ {\frac{r_i\left(A^{k+M}\right) r_j\left(A^{k+N}\right)}{r_i(A^k) \, r_j(A^k)}} : a_{ij}^{(M)}>0 \right].
\end{equation*}
Since $\rho(A^{M+N})=\rho(A)^{M+N}$, taking the $M+N$-th root of the inequality above yields the upper bound in \eqref{eqX1}.
The proof of the lower bound in \eqref{eqX1} is completely analogous.

We now show the equality condition, assuming that $A^{M+N}$ is irreducible. 
If equality holds in the upper bound of \eqref{eqX1}, then \eqref{2.4.1} must also hold with equality.
Then, by the equality condition of Theorem~\ref{thmF} applied to $D^{-1}A^{M+N} D$, we conclude that, in fact,
\begin{equation}
\label{2.4.7}
\rho(D^{-1}A^{M+N} D) =r_i \left(D^{-1}A^{M+N} D\right),
\end{equation}
for all $i\in\{1, \ldots, n\}$.
Referring to \eqref{usefulD}, we know that $r_i \left(D^{-1}A^{M+N} D\right)={(A^{M+N} x)_i}/{x_i}$, for all $i$,
where $x = (r_1(A^k), \ldots, r_n(A^k))^T$.
From \eqref{2.4.0} and \eqref{2.4.7}, we conclude that $\rho(A^{M+N}) ={(A^{M+N} x)_i}/{x_i}$, for all $i$.
This shows that $x$ is a positive eigenvector of $A^{M+N}$ corresponding to $\rho(A^{M+N})$. 
Since $A^{M+N}$ is irreducible, we may apply Lemma~\ref{irred-red} to show that $x$ is a positive eigenvector of $A$, as desired.
If equality holds in the lower bound of \eqref{eqX1}, the same conclusion is verified in an analogous manner.

Conversely, suppose that $x= (r_1(A^k), \ldots, r_n(A^k))^T$ is a positive eigenvector of $A$.
Then Lemma~\ref{lemEqEig} implies
\begin{equation*}
\label{2.4.10}
\frac{r_i(A^{k+t})}{r_i(A^k)} = \rho(A)^{t},
\end{equation*}
for all $i\in \{1, \ldots, n\}$ and $t\ge0$. 
This shows that both the upper and lower bounds in \eqref{eqX1} hold with equality. 
\end{proof}

\begin{rmk}
As was the case in Theorem~\ref{thmM1}, the nonzero row-sum assumption is not required for $k=0$.  
\end{rmk}

The result in \citet{Xu} was limited to the case where $k=M=N=1$ and, for just the lower bound, where $A$ is irreducible. 
Note that the upper (resp., lower) bound of Theorem~\ref{thmXuXuMat} degenerates to Theorem~\ref{thmM1} (with $L=M$)
when  $N=0$ or when $M=N$ and $a^{(M)}_{ii}>0$, where $i$ is the index of the row in A having the greatest (resp., least) row sum.
Since these new bounds and the bounds of Xu \& Xu depend upon the sparsity pattern of $A$, they may produce sharper bounds when critical entries of $A$ are zero.
However, from the proof it is clear that Theorem~\ref{thmM1} with $L=M+N$ is at least as tight as Theorem~\ref{thmXuXuMat}.
Nevertheless, in some applications it may be prohibitively complex to compute $r_i(A^{M+N+k})$ for Theorem~\ref{thmM1}
as opposed to examining the sparsity pattern of $A^M$ and computing $r_i(A^{M+k})$ and $r_i(A^{N+k})$ as required by Theorem~\ref{thmXuXuMat}.

Next, we briefly review bounds of a similar form developed by \citet[\textsection5]{Kolo}. 

\begin{thm}[Kolotilina]
\label{thmKolo}
Let $A=(a_{ij})$ be an $n \times n$ nonnegative matrix with spectral radius $\rho(A)$ and row sums $r_1(A),\ldots,r_n(A)$, all nonzero.  Then
\begin{equation*}
\min_{\substack{1\le i\le n\\1\le j\le n}}  \left[{r_i^\alpha(A) r_j^{1-\alpha}(A)} : a_{ij}>0 \right]
\le \rho(A) \le 
\max_{\substack{1\le i\le n\\1\le j\le n}} \left[{r_i^\alpha(A) r_j^{1-\alpha}(A)} : a_{ij}>0 \right]
\end{equation*}
for any $\alpha$ such that $0 \le \alpha \le 1$.
\end{thm}

For the proof and the equality conditions of Theorem~\ref{thmKolo}, see \citet[\textsection5]{Kolo}.
Applying Theorem~\ref{thmKolo} to $D^{-1}A^L D$, where integer $L\ge1$ and matrix $D=\diag(r_1(A^k),\ldots,r_n(A^k))$, yields
\begin{equation}
\label{eqKoloSum}
\min_{\substack{1\le i\le n\\1\le j\le n}} \left[ \frac{r_i^\alpha(A^{k+L}) r_j^{1-\alpha}(A^{k+L})}{r_i^\alpha(A^k) r_j^{1-\alpha}(A^k)} : a_{ij}^{(L)}>0 \right]
\le \rho(A^L) \le 
\max_{\substack{1\le i\le n\\1\le j\le n}} \left[ \frac{r_i^\alpha(A^{k+L}) r_j^{1-\alpha}(A^{k+L})}{r_i^\alpha(A^k) r_j^{1-\alpha}(A^k)} : a_{ij}^{(L)}>0 \right].
\end{equation}
Note that \eqref{eqKoloSum} with $\alpha = 0.5$ is equivalent to \eqref{eqX1} with $L=M=N$.

The theorems and corollaries presented in this section have been structured around row sums. 
Since $\rho(A^T)=\rho(A)$, similar bounds may be obtained starting with column sums.

\section{Further Equality Conditions on the Spectral Radius Bounds} 
\label{sectEQ1}

In this section we develop alternative equality conditions for the spectral bounds of the previous section by generalizing the proofs in \citet{Zhang}.
We provide detailed proofs for the equality conditions in which the more general expression \eqref{eqX1} holds.
The corollaries of this section treat \eqref{eqM1} as a special case of \eqref{eqX1}.
Like Zhang \& Li, we divide the equality conditions for these bounds into two cases corresponding to whether $A^L$ is irreducible or reducible. 
We address the former first.

\begin{thm}
\label{thmX2}
Let $A$ be an $n \times n$ nonnegative matrix with spectral radius $\rho(A)$ and nonzero row sums, and 
let $M>0$, $N\ge 0$, and $k\ge0$ be integers.
If $A^{M+N}$ is an irreducible matrix, then equality holds on either side (and hence both sides) of \eqref{eqX1} if and only if
\begin{equation}
\label{eqM2}
\frac{r_i(A^{k+1})}{r_i(A^{k})}=\rho(A), 
\end{equation}
for all $i \in \left\{ 1,\ldots,n \right\}$. 
\end{thm}
\begin{proof}
By the equality condition of Theorem~\ref{thmXuXuMat}, equality holds in \eqref{eqX1} 
if and only if $x=(r_1(A^k),\ldots,r_n(A^k))^T$ is an eigenvector of $A$.
Using \eqref{usefulx}, we may restate $(A x)_i = \rho(A) x_i$ as \eqref{eqM2}, for all $i \in \left\{ 1,\ldots,n \right\}$.
The converse follows using the prior argument in reverse. 

\end{proof}

\begin{cor}
\label{corM4}
Theorem~\ref{thmX2} also describes equality in \eqref{eqM1} when $A^{L}$ is irreducible for some $L>0$.
\end{cor}
\begin{proof}
The proof follows analogously, letting $N=0$ and $L=M$.
\end{proof}

The case in which $A^L$ is reducible requires some background concerning imprimitive matrices, which we review next.
A nonnegative irreducible matrix $A$ having only one eigenvalue with a modulus equal to $\rho(A)$ is said to be \emph{primitive}. 
If a nonnegative irreducible matrix $A$ has $h > 1$ eigenvalues with modulus $\rho(A)$, it is said to be \emph{imprimitive} 
or a cyclic matrix, and $h$ is known as the \emph{index of imprimitivity}.

\begin{lem}
\label{lemsuperdiag}
Let $A$ be an $n \times n$ irreducible nonnegative matrix with index of imprimitivity equal to $h$.
Let $L>0$ be an integer and $r$ be the greatest common divisor (gcd) of $h$ and $L$.
Then $A^L$ is reducible if and only if $r>1$.
In general there is a permutation matrix $P$ that symmetrically permutes $A^L$ to the block diagonal matrix 
\begin{equation}
\label{reducdiag}
P A^L P^T =\begin{pmatrix} 
C_{1}&0&\cdots&0\\
0&C_{2}&\cdots&0\\
\vdots&\vdots&\ddots&\vdots\\
0&0&\cdots&C_{r} \end{pmatrix},
\end{equation}
where each $C_\ell$ matrix is an $n_\ell \times n_\ell$ irreducible nonnegative matrix.
Furthermore, for $r>1$, $P$ also symmetrically permutes $A$ to form
\begin{equation}
\label{superdiag}
P A P^T =\begin{pmatrix} 0&A_{12}&0&\cdots&0\\
0&0&A_{23}&\cdots&0\\
\vdots&\vdots&\vdots&\ddots&\vdots\\
0&0&0&\cdots&A_{r-1,r}\\
A_{r,1}&0&0&\cdots&0 \end{pmatrix},
\end{equation}
where the all-zero submatrices along the diagonal are square and of order $n_1,\ldots,n_r$, respectively.
When \eqref{superdiag} holds with $r>1$, we say that $A$ is \emph{$r$-cyclic}.
The {block} (\textit{i.e.}, submatrix) $A_{\ell,m}$ is $n_\ell \times n_m$, for all $\ell \in\{1,\ldots,r\}$ and $m=(\ell \bmod r)+1$.
Moreover,
\begin{equation}
\label{Cs_def}
\begin{split}
C_1 &=\left[ A_{12}A_{23}\cdots A_{r-1,r}A_{r,1}\right]^{(L/r)}\\
C_2 &=\left[ A_{23}A_{34}\cdots A_{r,1}A_{12}\right]^{(L/r)}\\
\ldots \\
C_r &=\left[ A_{r,1}A_{12}\cdots A_{r-2,r-1}A_{r-,1r}\right]^{(L/r)},
\end{split}
\end{equation}
and $\rho(A^L)=\rho(C_1)=\cdots =\rho(C_r)$.
\end{lem}
\begin{proof}
See \citet[\textsection3.4]{Brualdi1991}.
\end{proof}
\begin{rmk}
Note that for a given matrix, its $r$ value may vary depending on the specified value of $L$, since $r=\gcd(h,L)$.
\end{rmk}

Recall from Perron-Frobenius theory that a square nonnegative matrix $A$, even a reducible one, has at least one nonnegative eigenvector $x\ne0$, 
such that $A x = \rho(A) x$  \cite[p. 503]{Horn}.
However, some reducible matrices, in fact, have a positive eigenvector and it need not be unique. 
In the case of the reducible matrix $A^L$ of Lemma~\ref{lemsuperdiag}, having eigenvalue $\rho(A^L)$ with an algebraic multiplicity equal to $r$, 
there are $r$ linearly independent positive eigenvectors of $A^L$ corresponding to $\rho(A^L)$, which we utilize to prove the next theorem.

\begin{thm}
\label{thmX3}
Let $A$ be an $n \times n$ irreducible nonnegative matrix with spectral radius $\rho(A)$, index of imprimitivity equal to $h$, and nonzero row sums.
Let $M>0$, $N\ge 0$, and $k\ge 0$ be integers and $r=\gcd(h,M+N)$.
If $A^{M+N}$ is reducible ($r>1$), then equality holds on either side (and hence both sides) of \eqref{eqX1}  if and only if 
\begin{equation}
\label{eqM4}
\frac{r_i(A^{k+1})}{r_i(A^{k})} =c_{m(i)},
\end{equation}
for all $i \in \{ 1,\ldots,n \}$,
where the $i$th row of $A$ has been assigned to the $\ell$th block, $\ell \in \{1,\ldots,r\}$, using the mapping $\ell=m(i)$ according to \eqref{reducdiag} and \eqref{superdiag},
and $c_\ell$ is a constant for the $\ell$th block.
Moreover, $\rho(A)^r = \prod_{\ell=1}^r c_\ell$.
\end{thm}
\begin{proof}
Let $x=(r_1(A^k),\ldots,r_n(A^k))^T$.
From the proof of Theorem~\ref{thmXuXuMat} we know that equality on the right side of \eqref{eqX1} implies 
\begin{equation*}
\rho(A^{M+N})= \max_{1\le i \le n} \left[\frac{(A^{M+N} x)_i}{x_i} \right].
\end{equation*}
Without loss of generality we will assume that $A$ is in the form of \eqref{superdiag} and $A^{M+N}$ is in block diagonal form \eqref{reducdiag}, where
each $C_\ell$ is $n_\ell \times n_\ell$ and irreducible.
Let $x$ be divided into $r$ subvectors, such that $x=(w_1^T,\ldots,w_r^T)^T$ and the $\ell$th subvector $w_\ell$ has $n_\ell$ elements.
Assuming that the $i$th entry of $x$ lies in the $\ell$th subvector $w_\ell$, where $\ell=m(i)$, we may form the general relation
\begin{equation*}
\frac{\bigl(A^{M+N} x\bigr)_i}{x_i} = 
\frac{\bigl(C_\ell w_\ell\bigr)_{j}}{(w_\ell)_{j}},
\end{equation*}
where $j=i - \sum_{t=1}^{\ell-1} n_t$ is the index within the $\ell$th subvector.
We may use Theorem~\ref{thmW} to bound the spectral radius of each $C_\ell$ as
\begin{equation}
\label{eqM6}
\rho(C_\ell) \le
\max_{1\le j \le n_\ell} \left[\frac{\bigl(C_\ell w_\ell\bigr)_{j}}{(w_\ell)_{j}} \right] \le
\max_{\begin{subarray}{l} 1\le j\le n_t\\
                       1\le t\le r\end{subarray} } \left[\frac{\bigl(C_t w_t\bigr)_{j}}{(w_t)_{j}} \right] = \rho(A^{M+N}),
\end{equation}
for all $\ell\in \{1,\ldots,r\}$.  
However, $\rho(A^{M+N}) = \rho(C_\ell)$ implies that equality holds throughout \eqref{eqM6} for all $\ell$.
Since $C_\ell$ is irreducible, we apply the equality condition of Theorem~\ref{thmW} to $C_\ell$ 
yielding $w_\ell$ as an eigenvector of $C_\ell$ associated with $\rho(C_\ell)$, for all $\ell\in \{1,\ldots,r\}$.
Therefore, $x$ is a positive eigenvector of $A^{M+N}$.

Now, we can show that 
\begin{equation*}
\set{X} = \left\{ (g_1 w_1^T, \ldots, g_r w_r^T)^T : g_1, \ldots, g_r \in \R^+ \right\},
\end{equation*}
is the complete set of positive eigenvectors of $A^{M+N}$, where $\R^+$ is the set of real, positive numbers.
If we evaluate $y$ in $A^{M+N} y = \rho(A^{M+N}) y$, where $y \in \set{X}$, it readily reduces to $g_\ell C_\ell w_\ell = \rho(A^{M+N}) g_\ell w_\ell$ or simply
$C_\ell w_\ell = \rho(A^{M+N}) w_\ell$, for all $\ell \in \{1,\ldots,r\}$ due to the block diagonal form of $A^{M+N}$.
Since $C_\ell$ is irreducible, $w_\ell$ is the unique positive eigenvector of $C_\ell$, up to a positive scale factor, for all $\ell \in \{1,\ldots,r\}$.
Hence, there can be no other positive eigenvectors of $A^{M+N}$ beyond set $\set{X}$.

Since $A$ is irreducible it must have a unique (up to a scale factor) positive eigenvector, we shall call $y_1$.
Additionally, $y_1$ is in the set $\set{X}$, because every eigenvector of $A$ must also be an eigenvector of $A^{M+N}$ and
$\set{X}$ is the entire set of positive eigenvectors of $A^{M+N}$.
Thus, we may simplify the ratio
\begin{equation}
\label{eqM7}
\frac{r_{i}(A^{k+1})}{r_{i}(A^{k})} = 
\frac{\left(A x\right)_{i}}{x_{i}} = 
\frac{\left(A_{12} w_2\right)_i}{\left(w_1\right)_i} = 
\frac{\rho(A) g_1 \left(w_1\right)_i}{g_2 \left(w_1\right)_i} = 
\rho(A) \frac{g_1}{g_{2}},
\end{equation}
when $1\le i \le n_1$,
where we take $g_1,\ldots,g_r$ to be real, positive constants that depend upon $y_1$.
Note that \eqref{eqM7} is constant within the first block.
The other blocks follow similarly, confirming \eqref{eqM4}.  
Starting with the equality on the left side of \eqref{eqX1} yields the same result.

Conversely, suppose that \eqref{eqM4} is true for all $i \in \{ 1,\ldots,n \}$. 
Then, for the first block ($\ell=1$), using \eqref{rowsum_order_k}, \eqref{usefulx}, and \eqref{eqM4},
\begin{equation*}
\frac{r_{i}(A^{k+L})}{r_{i}(A^{k})} = 
\frac{\sum_{j=1}^n a_{ij}^{(L-1)} r_j(A^{k+1})}{r_i(A^k)}=
\frac{c_1 \sum_{j=1}^n a_{ij}^{(L-1)} r_j(A^k)}{r_i(A^k)}=\cdots =
\frac{c_1 c_2 \cdots c_L \, r_i(A^k)}{r_i(A^k)} = c_1 c_2 \cdots c_L,
\end{equation*}
for all $1\le i \le n_1$ and $1\le L \le r$.
For values of $L$ greater than $r$ the indexing of $c_\ell$ must wrap around to $1$.
Thus, for the $i$th row, which is in block $m(i)$, we may form
\begin{align}
\label{eqX2}
\frac{r_{i}(A^{k+M})}{r_{i}(A^{k})} &= \prod_{\ell=m(i)}^{m(i)+M-1} c_{[(\ell-1)\bmod r]+1}
&\mbox{ and}&&  
\frac{r_{i}(A^{k+N})}{r_{i}(A^{k})} &= \prod_{\ell=m(i)}^{m(i)+N-1} c_{[(\ell-1)\bmod r]+1},
\end{align}
for all $i \in \{1,\ldots,n\}$, where we have performed modulo arithmetic on the indexing of $c_\ell$ terms, as required.
Recognize that the values of \eqref{eqX2} are still dependent on the row index $i$.  
Next, we limit our consideration of \eqref{eqX2} to rows $i$ and $j$, respectively, such that $a_{ij}^{(M)}>0$ as in \eqref{eqX1}.
Thus, row $j$ is in block $m(j)$, where $m(j)=[m(i)+M-1)\bmod r]+1$.
Therefore, forming the product of the row sum ratios in \eqref{eqX2}, with that restriction, yields
\begin{equation}
\label{eqX3}
\frac{r_{i}(A^{k+M})}{r_{i}(A^{k})} \frac{r_{j}(A^{k+N})}{r_{j}(A^{k})} =
\prod_{\ell=m(i)}^{m(i)+M-1} c_{[(\ell-1)\bmod r]+1}  
\prod_{t=m(i)+M}^{m(i)+M+N-1} c_{[(t-1)\bmod r]+1} =
\left( \prod_{\ell=1}^{r} c_{\ell} \right)^{(M+N)/r}\mspace{-20.0mu},
\end{equation}
for all $i \in \{1,\ldots,n\}$ and all $j$ such that $a_{ij}^{(M)}>0$.
That is, the product of constants in \eqref{eqX3} includes exactly $(M+N)/r$ occurrences of each constant $c_\ell$.
This results in a value for \eqref{eqX3} that is independent of $i$ and $m(i)$ and, hence, equality is true on both sides of \eqref{eqX1}
with $\rho(A)^r = \prod_{\ell=1}^r c_\ell$.

\end{proof}

\begin{cor}
\label{corM4a}
Theorem~\ref{thmX3} also describes equality in \eqref{eqM1} when $A$ is irreducible and $A^{L}$ is reducible for some $L>0$.
\end{cor}
\begin{proof}
The proof follows analogously, letting $N=0$ and $L=M$.
\end{proof}

\section{Digraphs and Sinks} 
\label{sectDigraph}
In this section, we cast the results of previous sections in graph-theoretic terms.  We first review some basic concepts and terminology related to digraphs. (For a more complete treatment, we refer the reader to  \citet[\textsection4.3]{Minc} and \citet[chap.~3]{Brualdi1991}.)
Let $G = (V, E)$ be a directed graph or \emph{digraph} with a nonempty set of $n$ vertices, $V = \{v_1, . . . , v_n\}$, 
and a collection $E$ of directed edges or \emph{arcs}.  
The digraph is called \emph{simple} if it contains no self-loops (arcs with the same initial and terminal vertex) or multiarcs (arcs 
that share the same initial and terminal vertices).  
We remark that, in contrast to the results of \citet{Zhang}, \citet{Xu}, and \citet{Gungor} that we generalize below, our results are not limited to simple digraphs.

Recall that the \emph{adjacency matrix} $A(G)$ of any digraph $G$ is the nonnegative matrix whose $(i,j)$-th entry $a_{ij}$ is the number of arcs directed  from vertex $v_i$ to vertex $v_j$ in $G$.
The adjacency matrix entry along the diagonal $a_{ii}$ is the number of self-loops at vertex $v_i$.
The \emph{spectral radius} $\rho(G)$ of digraph $G$ is defined to be the spectral radius of $A(G)$.
In a digraph $G$, a \emph{directed walk} is an alternating sequence of vertices and arcs from $v_i$ to $v_j$ in $G$ such that
every arc in the sequence is preceded by its initial vertex and is followed by its terminal vertex. 
The \emph{length} of a directed walk is the number of arcs in the sequence, which must be one or more.
The number of distinct directed walks from $v_i$ to $v_j$ of length-$k$ in $G$ is equal to the $(i,j)$-th entry of $A(G)^k$.
The digraph $G$ is \emph{strongly connected} if and only if $A(G)$ is irreducible.
A strongly connected digraph $G$ is also characterized by an \emph{index of imprimitivity} $h(G)$ which is equal to the index of imprimitivity of $A(G)$.
Furthermore, a digraph $G$ is classified as \emph{cyclically r-partite} when $r>1$ and $r$ divides $h(G)$,  \citet[\textsection3.4]{Brualdi1991}.

The \emph{outdegree} $d^+_i$ of vertex $v_i\in V$ in the digraph $G=(V,E)$ is defined to be the number of arcs in $E$ with initial vertex $v_i$. 
Thus, the outdegree of vertex $v_i$ is equal to the $i$th row sum of the adjacency matrix $A(G)$.
This concept can be generalized to the \emph{$k$-outdegree} $d^{k+}_i$, which is the number of directed walks of length $k$ with initial vertex $v_i$.
That is, for $A(G)=(a_{ij})$,
\begin{align*}
d^+_i &\triangleq \sum_{j=1}^n a_{ij},&
d_i^{k+} &\triangleq \sum_{j=1}^n a_{ij}^{(k)},
&\mbox{ and}&&
d^{0+}_i &\triangleq 1.
\end{align*}

In a digraph, a vertex $v_i$ with no outgoing arcs (\textit{i.e.}, $d^+_i=0$) is known as a \emph{sink}.
Sinks correspond to zero rows in $A(G)$.
Thus, the results of Sections~\ref{sectMatrix} and \ref{sectEQ1} directly apply to general digraphs without sinks.
The two theorems in this section bound the spectral radii of these digraphs, while
the corollaries will show that equality in the theorems may only be achieved for very limited values of $\rho(G)$.
We complete this section with a detailed example.

We need to introduce terminology to capture the equality conditions of Section~\ref{sectEQ1} in a digraph context.
With respect to Theorem~\ref{thmX2}, we will call digraph $G=(V,E)$ \emph{average $\kappa$-outdegree regular} if 
\begin{equation*}
\frac{d_i^{\kappa+}}{d_i^{(\kappa-1)+}}=c, \text{ for all } v_i \in V,
\end{equation*}
where $\kappa\ge 1$. 
Thus for $\kappa=2$ our definition matches that of Zhang \& Li.
If $G$ is cyclically $r$-partite, the set of vertices $V$ may be partitioned into $r$ disjoint subsets $V=V_1 \cup V_2 \cup \cdots \cup V_r$ according to \eqref{superdiag}.
With respect to Theorem~\ref{thmX3}, we will call digraph $G$ \emph{average $\kappa$-outdegree $r$-quasiregular} if 
$G$ is cyclically $r$-partite and 
\begin{equation*}
\frac{d_i^{\kappa+}}{d_i^{(\kappa-1)+}}=c_j, \text{ for all } v_i \in V_j, 
\end{equation*}
where $\kappa\ge 1$. 
To be cyclically $r$-partite, all arcs joining vertices in $V_j$ either initiate in $V_{[(j-2) \bmod r]+1}$ or terminate in $V_{(j \bmod r)+1}$.
Thus for $r=2$, this condition degenerates to the bipartite-semiregular condition of Zhang \& Li.
For $\kappa =1$, we will typically drop the word ``average'' from these two new terms to be consistent with prior terminology.

\begin{thm}
\label{thmGLiu}
Let $G$ be a digraph with spectral radius $\rho(G)$, $n$ vertices, and no sinks. 
Then, for any integers $L>0$ and $k\ge0$, 
\begin{equation}
\label{eqGLiu}
\min_{1\le i\le n}  \left({\frac{d_i^{(k+L)+}}{d_i^{k+}}}\right)^{1/L}\mspace{-10.0mu} \le \rho(G) \le 
\max_{1\le i\le n}  \left({\frac{d_i^{(k+L)+}}{d_i^{k+}}}\right)^{1/L}\mspace{-10.0mu}.
\end{equation}
Moreover, if $G$ is strongly connected, then any equality in \eqref{eqGLiu} holds if and only if
$G$ is average $(k+1)$-outdegree regular or average $(k+1)$-outdegree $r$-quasiregular or both,
where $r=\gcd(L,h(G))$ and $h(G)$ is the index of imprimitivity of $G$.
\end{thm}
\begin{proof}
Applying Theorem~\ref{thmM1} to the adjacency matrix $A(G)$, we derive \eqref{eqGLiu}.
The equality conditions are justified by Corollaries~\ref{corM4} and \ref{corM4a}. 
\end{proof} 

Expression \eqref{eqGLiu} is the digraph equivalent of the main result of \citet{Liu}, but the equality conditions of Theorem~\ref{thmGLiu} are new.
Special cases of Liu's result have been rediscovered by several recent works.
With $(k,L)=(0,2)$ and $(1,2)$ Liu's result appears in \citet{Zhang} for simple, strongly connected digraphs. 
Also, the bounds of \citet{Gungor} include the bounds of Theorem~\ref{thmGLiu}, with $L=1$ and $2$, 
applied to digraphs that are simple and strongly connected.

\begin{lem}
\label{conj1}
Let $c$ be a rational number and $d$ be a nonzero, finite real number.
If the sequence $\{d \cdot c^j\}_{j=0}^\infty$ contains only integers, then $c$ is also an integer.
\end{lem}
\begin{proof}
Let $c= p/q$, where $p$ and $q>0$ are coprime integers.
Since $d \cdot p^j / q^j$ is an integer, $d$ is a multiple of $q^j$, 
which cannot hold for all $j\ge 0$ unless $q=1$.
\end{proof}

\begin{cor}
\label{corGLiuEq}
Let $G$ be a strongly connected digraph.
If equality holds in \eqref{eqGLiu} for $G$ and some $k=t$, then equality holds for all $k\ge t$ and
$\rho(G)$ is the $r$th root of an integer, where $r=\gcd(L,h(G))$. 
\end{cor}
\begin{proof}
Theorem~\ref{thmMono} proves the first part by showing that the bounds are monotonic in $k$.
In the case that $A^L$ is irreducible $(r=1)$ and equality holds in \eqref{eqGLiu}, then by Corollary~\ref{corM4}, $c=\rho(G)$ is a rational number
and $d_i^{(j+t)+}=d_i^{t+} \rho(G)^{j}$ holds for all $j \ge0$ and all $i \in \{1,\ldots,n\}$.
Since the $k$-outdegree $d_i^{k+}$ of any vertex is integral, then by Lemma~\ref{conj1}, $\rho(G)$ is an integer.

In the case that $A^L$ is reducible $(r>1)$ and equality holds in \eqref{eqGLiu}, then by Corollary~\ref{corM4a}, $\rho(G)^r$ is a rational number
and $d_i^{(r j+t)+}=d_i^{t+} \rho(G)^{r j}$ holds for all $j \ge0$ and all $i \in \{1,\ldots,n\}$.
Thus,  $\rho(G)^r$ is an integer.
\end{proof}

Next, we apply our generalized spectral bounds of Xu \& Xu, which appear herein as Theorem~\ref{thmXuXuMat}, to digraphs.
We also formulate equality conditions similar to those in the preceding theorem.

\begin{thm}
\label{thmGNew}
Let $G=(V,E)$ be a digraph with spectral radius $\rho(G)$, $n$ vertices, and no sinks.  
Then, for any nonnegative integers $N$ and $k$, 
\begin{equation}
\label{eqGXu}
\min_{\substack{1\le i\le n\\1\le j\le n}} \left[\left({\frac{d_i^{(k+1)+} d_j^{(k+N)+}}{d_i^{k+}\, d_j^{k+}}} \right)^\frac{1}{N+1}\mspace{-12.0mu} : (v_i,v_j)\in E \right] 
\le \rho(G) \le 
\max_{\substack{1\le i\le n\\1\le j\le n}} \left[\left({\frac{d_i^{(k+1)+} d_j^{(k+N)+}}{d_i^{k+}\, d_j^{k+}}} \right)^\frac{1}{N+1}\mspace{-12.0mu} :  (v_i,v_j)\in E \right].
\end{equation}
Moreover, if $G$ is strongly connected, then equality in \eqref{eqGXu} holds if and only if
$G$ is average $(k+1)$-outdegree regular or average $(k+1)$-outdegree $r$-quasiregular or both,
where $r=\gcd(N+1,h(G))$ and $h(G)$ is the index of imprimitivity of $G$.
\end{thm}
\begin{proof}
We apply Theorem~\ref{thmXuXuMat} with $M=1$ to the adjacency matrix $A(G)$ of digraph $G$ to yield \eqref{eqGXu}.
Theorems~\ref{thmX2} and \ref{thmX3} justify the equality conditions.
\end{proof}
\begin{rmk}
We limit this theorem to $M=1$ for simplicity in expressing \eqref{eqGXu}.
\end{rmk}

A special case of \eqref{eqGXu} with $(k,N)=(1,1)$ appears in \citet{Xu} for simple digraphs that are, in the case of the lower bound, strongly connected.
The proofs within \cite{Xu} that are required for the $(1,1)$ bound appear easily generalizable to digraphs with self-loops and multiarcs.
The bounds of G\"{u}ng\"{o}r \& Das include the bounds of Theorem~\ref{thmGNew} with $N=1$,
but they are limited to simple, strongly connected digraphs.
The other realizations of \eqref{eqGXu} are new and may be tighter than the prior bounds. 


\begin{cor}
\label{corGXuEq}
Let $G$ be a strongly connected digraph.
If equality holds in \eqref{eqGXu} for $G$ and some $k\ge0$ and $N\ge 0$, then 
$\rho(G)$ is the $r$th root of an integer, where $r=\gcd(N+1,h(G))$. 
\end{cor}
\begin{proof}
As discussed in Section~\ref{sectMatrix}, the bounds of Theorem~\ref{thmM1} with $L=M+N$ are at least as tight
as the bounds of Theorem~\ref{thmXuXuMat}.  
Therefore, with respect to the upper bounds,
\begin{equation}
\label{eqGXutight}
\rho(G)^{N+1} \le 
\max_{1\le i\le n}  \left[{\frac{d_i^{(k+N+1)+}}{d_i^{k+}}}\right]  \le
\max_{\substack{1\le i\le n\\1\le j\le n}} \left[{\frac{d_i^{(k+1)+} d_j^{(k+N)+}}{d_i^{k+}\, d_j^{k+}}}  :  (v_i,v_j)\in E \right].
\end{equation}
Assuming equality holds on the right side of \eqref{eqGXu}, then equality holds throughout \eqref{eqGXutight}.
Since $G$ is strongly connected, we may apply Corollary~\ref{corGLiuEq} to the left equality in \eqref{eqGXutight}, proving
that $\rho(G)$ is the $r$th root of an integer.  
Starting with equality on the left side of \eqref{eqGXu}, yields that same result.
\end{proof}

Just as deleting the zero rows and their corresponding columns preserved the spectral radius in Section~\ref{sectMatrix},
the removal of any sinks from the digraph $G$ leaves $\rho(G)$ undisturbed.  
Thus, this simple modification allows us to extend the bounds of this section to general digraphs.
Additionally, removing sources from the digraph may tighten the bounds.

\begin{figure}[h!]
\centering
\includegraphics[width=1.37 in]{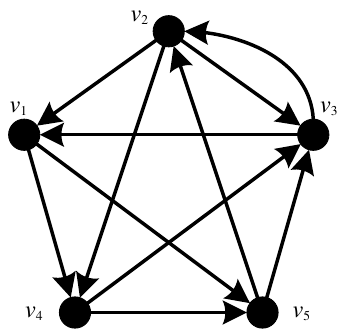}
\caption{Digraph $G_1$ having adjacency matrix $A(G_1)$}
\label{fig1}
\end{figure}

\begin{exmp}
\label{Ex2}
The order-$5$ example presented in \cite{Xu} and \cite{Gungor} provides a useful illustration. 
Given the digraph $G_1=(V,E)$ shown in Fig.~\ref{fig1}, we find the $5\times5$ adjacency matrix to be
\begin{equation*}
A(G_1)= 
\begin{pmatrix}
0&0&0&1&1\\
1&0&1&1&0\\
1&1&0&0&0\\
0&0&1&0&1\\
0&1&1&0&0
\end{pmatrix}.
\end{equation*}
The spectral radius of $G_1$ is $\rho(G_1) \approx 2.193399638$. 
First, we examine Theorem~\ref{thmGLiu}.
Table~\ref{Table1} shows the 
quantities corresponding to each vertex $v_i \in V$ needed to evaluate \eqref{eqGLiu},
for all values of $(k,L)$ such that $L+k\le 4$.
The minimum and maximum of each, shown on the right side of the table, form the bounds on $\rho(G_1)$.
\renewcommand{\arraystretch}{1.1} 
\begin{table}[h!]
\caption{Intermediate Computations and Bounds of Theorem~\ref{thmGLiu} for all $(k,L)$, such that $L+k \le 4$.}
\label{Table1}
\begin{center}
\begin{tabular}{r||l|l|l|l|l||l|l}
\hline
{\bfseries $(k,L)$ Bound} &  &  &  &  &  &  & \\
{\bfseries Parameters} & {$i=1$} & {$i=2$} & {$i=3$} & {$i=4$} & {$i=5$} & {$\min$} & {$\max$}\\
\hline\hline
{\bfseries $d_i^+              \quad (0,1)$} &	2 &	3 &	2 &	2 &	2 &	2 &	3 \\
{\bfseries $\sqrt{d_i^{2+}}\quad (0,2)$} &	2 &	2.4495 &	2.2361 &	2 &	2.2361 &	2 &	2.4495 \\
{\bfseries $d_i^{2+}/d_i^{1+}\quad  (1,1)$} & 	2 &	2 &	2.5 &	2 &	2.5 &	2 &	2.5 \\  
\hline\noalign{\smallskip}
{\bfseries $\sqrt[3]{d_i^{3+}}\quad  (0,3)$} &	2.0801 &	2.3513 &	2.1544 &	2.1544 &	2.2240 &	2.0801 &	2.3513 \\
{\bfseries $\sqrt{d_i^{3+}/d_i^{1+}}\quad  (1,2)$} &	2.1213 &	2.0817 &	2.2361 &	2.2361 &	2.3452 &	2.0817 &	2.3452 \\
{\bfseries $d_i^{3+}/d_i^{2+}\quad  (2,1)$} & 	2.25 &	2.1667 &	2 &	2.5 &	2.2 &	2 &	2.5 \\
\hline\noalign{\smallskip}
{\bfseries $\sqrt[4]{d_i^{4+}}\quad  (0,4)$} &	2.1407 &	2.3206 &	2.1657 &	2.1407 &	2.1899 &	2.1407\textdagger &	2.3206 \\
{\bfseries $\sqrt[3]{d_i^{4+}/d_i^{1+}}\quad  (1,3)$} & 	2.1898 &	2.1302 &	2.2240 &	2.1898 &	2.2572 &	2.1302 &	2.2572\textdagger \\
{\bfseries $\sqrt{d_i^{4+}/d_i^{2+}}\quad  (2,2)$} & 	2.2913 &	2.1985 &	2.0976 &	2.2913 &	2.1448 &	2.0976 &	2.2913 \\
{\bfseries ${d_i^{4+}/d_i^{3+}}\quad  (3,1)$} & 	2.3333 &	2.2308 &	2.2 &	2.1 &	2.0909 &	2.0909 &	2.3333 \\
\hline
\end{tabular}
\end{center}
\end{table}
The bound corresponding to $(k,L)=(1,2)$ is the tightest of the bounds here for $L+k \le 3$.
When extended to $L+k=4$, the bounds using $(k,L)=(0,4)$ and $(1,3)$ yield the tightest lower and upper bounds, respectively, as indicated with a ``\textdagger''.

\renewcommand{\arraystretch}{1.1} 
\begin{table}[h!]
\caption{Lower and Upper Bounds on  $\rho(G_1)$ from Theorem~\ref{thmGNew}.}
\label{Table2}
\begin{center}
\begin{tabular}{r||l|l}
\hline
{\bfseries $(k,N)$ Bound} & {\bfseries Lower Bound }   & {\bfseries Upper Bound }\\
{\bfseries Parameters} &      {\bfseries on $\rho(G_1)$} & {\bfseries on $\rho(G_1)$}\\
\hline\hline
{\bfseries $(0,1) $ }&2 & 2.4495 \\
{\bfseries $(0,2) $ }&2 & 2.4662 \\
{\bfseries $(1,1) $ }&2 & 2.5 \\
\hline
{\bfseries $(0,3) $ }&2.0598 & 2.3403\textdaggerdbl \\
{\bfseries $(1,2) $ }&2.0801 & 2.3208\textdaggerdbl \\
{\bfseries $(2,1) $ }&2.0817 & 2.3717 \\
\hline
{\bfseries $(0,4) $ }&2.1118 & 2.3116 \\
{\bfseries $(1,3) $ }&2.1407 & 2.2900 \\
{\bfseries $(2,2) $ }&2.1204 & 2.2774 \\
{\bfseries $(3,1) $ }&2.0954 & 2.2815 \\ 
\hline
\end{tabular}
\end{center}
\end{table}

Theorem~\ref{thmGNew} yields the bounds shown in Table~\ref{Table2}.
In three of four cases the bounds of Theorem~\ref{thmGNew} with $N=1$ produced tighter bounds than Theorem~\ref{thmGLiu} with $L=1$. 
Also, the bounds indicated with a ``\textdaggerdbl'' are tighter than the bounds of the first table for the same maximum order of outdegree computed.

\renewcommand{\arraystretch}{1.1} 
\begin{table}[h!]
\caption{Lower and Upper Bounds on  $\rho(G_1)$ from \eqref{eqKoloSum}.}
\label{Table3}
\begin{center}
\begin{tabular}{r||l|l}
\hline
{\bfseries $(k,L)$ Bound} & {\bfseries Lower Bound }   & {\bfseries Upper Bound }\\
{\bfseries Parameters}               &      {\bfseries on $\rho(G_1)$} & {\bfseries on $\rho(G_1)$}\\
\hline\hline
{\bfseries $(0,1)$}&2*                              &2.4495@$\alpha=0.50$ \\
{\bfseries $(1,1)$}&2*                              &2.5000* \\
{\bfseries $(2,1)$}&2.0801@$\alpha=0.50$ &2.3602@$\alpha=0.55$ \\
{\bfseries $(3,1)$}&2.0993@$\alpha=0.92$ &2.2611@$\alpha=0.70$ \\
\hline
\end{tabular}
\end{center}
\end{table}

Finally, we show the Kolotilina-based bounds of \eqref{eqKoloSum} in Table~\ref{Table3}.
Recall that Theorem~\ref{thmGNew} limited the bounds of Theorem~\ref{thmXuXuMat} to the case where $M=1$ to
keep the sparsity pattern logic simple in the terminology of graphs.
Likewise, in evaluating \eqref{eqKoloSum}, we limit consideration of $L$ to $1$.
In Table~\ref{Table3} we have used an ``*'' to indicate which bounds are independent of the $\alpha$ parameter.
The best lower bounds of Theorem~\ref{thmGNew} were equal to those produced by Theorem~\ref{thmGLiu} but tighter 
than those produced by \eqref{eqKoloSum}.
Also, Theorem~\ref{thmGNew} produced tighter upper bounds compared with either \eqref{eqKoloSum} or Theorem~\ref{thmGLiu} 
when the maximum order of outdegree was limited to three.
However, Theorem~\ref{thmGNew} produced slightly looser upper bounds when the maximum order of outdegree was limited to four.

In this example, the worst set of $(k,N)$ parameters for Theorem~\ref{thmGNew} is from \citet{Xu} (\textit{i.e.,} $(1,1)$).
However, we have found that the best set of parameters depends on the digraph selected.
For digraphs that are sparser than $G_1$, the advantages of Theorem~\ref{thmGNew} and \eqref{eqKoloSum} will be even more evident.

We note that the bipartite condition (\textit{i.e.}, cyclically $r$-partite with $r=2$ herein) was sometimes unmentioned in \citet{Xu} and \citet{Gungor} 
when defining their semiregular digraph condition.
Its necessity is apparent in this example.
The digraph $G_1$ might be outdegree semiregular and average 2-outdegree semiregular by such looser definitions, 
but it is not bipartite and hence does not meet the bounds with equality.
\end{exmp}


\section{Conclusions} %


We have generalized the bounds and equality conditions of several prior works regarding the spectral radius of digraphs.
Much of the earlier work applied to irreducible matrices and strongly-connected simple digraphs.
We have generalized these to a larger set of bounds and a more general set of digraphs.

We have also cited the contributions of \citet{Liu}, missing in more recent works, and added digraph related equality conditions to Liu's bounds.
Our generalization of the bounds by Xu \& Xu are novel and sometimes outperform the prior bounds.
Finally, we have shown that the equality conditions of the bounds, when applied to strongly connected digraphs, may only be met 
when the spectral radius is the $r$th root of an integer.

\section*{Acknowledgment} 
The authors would like to thank the anonymous reviewer for greatly simplifying the proof of Theorem~\ref{thmXuXuMat}
and providing other insight which improved this work.



\end{document}